\documentclass[11pt]{amsart}
\usepackage{amsmath,amssymb,amsthm,amscd,verbatim}
\usepackage{graphicx}
\usepackage{epsfig}
\usepackage{stmaryrd}

\input{epsf}

\newtheorem{prelem}{{\bf Theorem}}

\newtheorem{theorem}{Theorem}[section]
\newtheorem{corollary}[theorem]{Corollary}
\newtheorem{definition}[theorem]{Definition}

\newtheorem{lemma}[theorem]{Lemma}
\newtheorem{proposition}[theorem]{Proposition}

\newtheorem{remarka}[theorem]{Remark}
\newenvironment{remark}{\begin{remarka}\rm}{\hfill\rule{2mm}{2mm}\end{remarka}}
\newtheorem{examplea}[theorem]{Example}
\newenvironment{example}{\begin{examplea}\rm}{\hfill\rule{2mm}{2mm}\end{examplea}}

\def\mod {{\rm mod}\ }

\def\Ex {{\mathbb E}}
\def\Pr {{\rm Pr}}

\title{A structure theorem for Boolean functions with small total influences}
\author{Hamed Hatami}
\address{School of Computer Science, McGill University, Montr\'eal,  Canada}
\email{hatami@cs.mcgill.ca }

\date{}

\begin{document}

\maketitle

\begin{center}
\emph{Dedicated to the memory of Avner Magen}
\end{center}

\begin{abstract}
We show that on every product probability space, Boolean functions with small total influences are essentially the ones that are almost measurable with respect to certain natural sub-sigma algebras. This theorem in particular describes the structure of monotone set properties that do not exhibit sharp thresholds.

Our result generalizes the core of Friedgut's seminal work~[Ehud Friedgut.
\newblock Sharp thresholds of graph properties, and the {$k$}-sat problem.
\newblock {\em J. Amer. Math. Soc.}, 12(4):1017--1054, 1999.] on properties of random graphs to the setting of arbitrary Boolean functions on general product probability spaces, and improves the result of Bourgain in his appendix to Friedgut's paper.
\end{abstract}
\newline
{{\sc Keywords:} influence, threshold, phase transition, Boolean function.

\vspace{0.5cm}


\section{Introduction \label{sec:intro}}

We call a function \emph{Boolean} if its range is $\{0,1\}$. The influence of a variable on a Boolean function measures the sensitivity of the function with respect to the changes in that variable. This basic notion  arises naturally in many different areas such as statistical physics and probability theory (e.g. phase transition, percolation), computer science (e.g. complexity theory, hardness of approximation, machine learning), combinatorics (e.g. set systems, products of graphs), economics (e.g. social choice). In many instances, when a Boolean function  satisfies some nice properties or fits a description, it is possible to bound the influences of its variables. Therefore, Boolean functions with small total influences arise frequently in  various contexts, and they are studied for different purposes.

One of the major motivations for studying Boolean functions with small total influences  is a profound connection to the \emph{threshold phenomenon}, discovered by Margulis~\cite{MR0472604} and Russo~\cite{MR618273}. The threshold behavior is the quick transition of a property  from being very unlikely to hold to being very likely to hold as certain parameter $p$ increases. This behavior occurs in various settings, and it is an instance of the phenomenon of phase transition in statistical physics which explains the rapid change of behavior in many physical processes. One of the main questions that arises in studying phase transitions is: ``How sharp is the threshold?'' That is how short is the interval in which the transition occurs. Margulis~\cite{MR0472604} and Russo~\cite{MR618273} observed that the sharpness of the threshold  is controlled by the total influence of the indicator function of the property. Hence in order to characterize the properties that do not exhibit sharp thresholds, one needs to understand the structure of Boolean functions that have small total influences.

Due to these motivations, Boolean functions with small total influences are studied extensively, and some remarkable results about their structure are discovered  e.g. by Kalai, Kahn and Linial~\cite{KKL}, Talagrand~\cite{MR1303654}, Bourgain and Kalai~\cite{BK}, Friedgut~\cite{MR1645642,MR1678031}, and Bourgain~\cite{Bourgain99}. The KKL inequality~\cite{KKL} and Friedgut's threshold theorem~\cite{MR1678031} are both mentioned in~\cite{Bourgain2000} as notable consequences of the interaction between harmonic analysis and combinatorics. Let us also mention the more recent work of Mossel, O'Donnell, and Oleszkiewicz~\cite{MR2630040} which studies Boolean functions where all variables have small influences, and now is one of the main tools in the study of the hardness of approximation in theoretical computer science.

The purpose of this article is to essentially characterize Boolean functions with small total influences by showing that every such function is almost  measurable with respect to a certain sub-sigma algebra.

\subsection{Notations and definitions}

For every statement $P$, we define $1_{[P]}:=1$ if $P$ is true, and $1_{[P]}:=0$ otherwise. For every natural number $n$, denote $[n]:=\{1,\ldots, n\}$, and for every set $S$, let $\mathcal{P}(S)$ denote the set of all subsets of $S$. Consider a probability space $X=(\Omega,\mathcal{F},\mu)$, and let $X^n$ denote the corresponding product space endowed with the product probability measure $\mu^n$. Throughout this article $X$ is always a probability space and $n$ is a positive integer, and all asymptotics  are meant for $n \rightarrow \infty$.

As usual, we use $O(X)$ to denote a quantity bounded in magnitude by $CX$ for some absolute constant $C$; if we need $C$ to depend on a parameter, we will indicate this by subscripts. Thus for instance $O_I(1)$ is a quantity bounded in magnitude by some expression $C_I$ depending on $I$. We use $o(X)$ to denote a quantity $Y$ with $\lim_{n \rightarrow \infty} Y/X=0$.

For every $x=(x_1,\ldots,x_n) \in X^n$, let $x_S:=(x_i: i \in S) \in X^S$  denote the restriction of $x$ to the coordinates in $S$. For two disjoint sets $S,T \subseteq [n]$, and elements $x \in X^S$ and $y \in X^T$, let $(x,y)$ denote the unique element $z \in X^{S \cup T}$ with $z_S=x$ and $z_T=y$.

Consider a subset $S \subseteq [n]$. In the sequel, by a slight abuse of notation, we will view functions on $X^S$ as also being functions on $X^n$. More precisely, we identify every function $g:X^S \rightarrow \mathbb{R}$ with the function on $X^n$ that maps every $x \in X^n$ to $g(x_S)$.

Let $f:X^n \rightarrow \{0,1\}$ be a  measurable function. For $1 \le j \le n$, the \emph{influence} of the $j$-th variable on $f$ is defined as
\begin{equation}
\label{eq:infDef}
I_f(j) = \Pr[f(x_1,\ldots,x_{j-1},x_j,x_{j+1},\ldots,x_n) \neq f(x_1,\ldots,x_{j-1},y_j,x_{j+1},\ldots,x_n)],
\end{equation}
where $x_1,\ldots,x_n,y_1,\ldots,y_n$ are i.i.d. random variables taking values in $X$ according to its probability measure.  The \emph{total influence} of $f$, denoted by $I_f$, is defined as
\begin{equation}
\label{eq:totalDef}
I_f := \sum_{j=1}^n I_f(j).
\end{equation}

\begin{remark}
\label{rem:finiteness}
Let $X=(\Omega,\mathcal{F},\mu)$, and $f:X^n \to \{0,1\}$ be measurable. It follows from the measurability of $f$ that the set $$\{(x_1,\ldots,x_n,y_j) : f(x_1,\ldots,x_{j-1},x_j,x_{j+1},\ldots,x_n) \neq f(x_1,\ldots,x_{j-1},y_j,x_{j+1},\ldots,x_n)\} \subseteq X^{n+1}$$ is measurable and thus $I_f(j)$ is well-defined. Moreover in many situations, to prove an statement about influences, one can assume that  $\Omega$ is a finite set. Indeed since $f$ is measurable, it is possible to find a \emph{finite} sub-$\sigma$-algebra $\mathcal{G}$ of $\mathcal{F}$ and a  function $g:\Omega^n \to \{0,1\}$, measurable with respect to the product $\sigma$-algebra generated by  $\mathcal{G}$, such that $\Pr[f(x) \neq g(x)]$ is arbitrarily small. Since $|I_f(j)-I_g(j)| \le 2 \Pr[f(x) \neq g(x)]$, the differences between the influences are also small.
\end{remark}

A particular case of interest is when $X$ is defined by the Bernoulli distribution $\mu_p$ on $\{0,1\}$ with parameter $0<p<1$, i.e. $\mu_p(\{1\})=p$ and $\mu_p(\{0\})=1-p$. We refer to the corresponding product probability distribution $\mu_p^n$ on $\{0,1\}^n$ as the \emph{$p$-biased distribution}.

A function $f:\{0,1\}^n \rightarrow \{0,1\}$ is called \emph{increasing} if $f(x) \le f(y)$, for every $x,y \in \{0,1\}^n$ satisfying $x_i \le y_i$ for every $i \in [n]$. For a set $A \subseteq [n]$, we say that a function $f:X^n \rightarrow \mathbb{R}$ depends only on the coordinates in $A$, if $f(x)=f(y)$ for every $x,y \in X^n$ with $x_A=y_A$.

\section{Main results \label{sec:main}}

Margulis~\cite{MR0472604} and Russo~\cite{MR618273} observed that for every increasing function $f:\{0,1\}^n \rightarrow \{0,1\}$, we have
\begin{equation}
\label{eq:MargulisRusso}
2p(1-p)\frac{d \mu_p(f)}{dp}=I_f,
\end{equation}
where $\mu_p(f) := \int f(x) d\mu_p^n(x)$ and $I_f$ is defined according to the $p$-biased distribution. This shows that the rate  at which  $\mu_p(f)$ increases  with respect to the increase in $p$ is controlled by $I_f$.
Consequently, in order to characterize the properties that do not exhibit sharp thresholds,  one needs to understand the structure of Boolean functions that have small total influences. We refer the reader to \cite{MR1678031} for more details.

The problem of finding general conditions under which a sharp threshold does not occur is first investigated by Russo~\cite{MR618273,MR671248}. Later, Talagrand~\cite{MR1303654}, extending the work of Russo, showed that for $p$ that is not too small, if the total influence of $f$ is small, then there are variables with large influences.

In the setting of the $p$-biased distribution, for sufficiently large $p$, the works of Talagrand~\cite{MR1303654}, Friedgut and Kalai~\cite{MR1371123}, Bourgain and Kalai~\cite{BK}, and Friedgut~\cite{MR1645642} provide a good understanding of the situation. Intuitively these results say that the total influence of $f$ is large, unless the value of $f(x)$ is determined only by ``local information'' about $x$, e.g. by a few number of coordinates.  As the following simple example shows, these results turn out however not to be useful when $p$ is small, in particular when $\log \frac{1}{p} \sim \log n$, often the case in applications.

\begin{example}
\label{ex:simple}
Set $p=n^{-1}$, and let $f:\{0,1\}^n \rightarrow \{0,1\}$ be defined as $f(x)=1$ if and only if $x \neq (0,\ldots,0)$. Then
$I_f(1)=\ldots=I_f(n) \le 2 p$, and so $I_f \le 2$. However note that there is no variable with large influence, and also $f$ does not depend only on a small set of its coordinates.
Indeed for every constant size set $A \subseteq [n]$, we have $$\Ex[f(x) | x_A=(0,\ldots,0)]=1-(1-p)^{n-|A|} = 1-\frac{1}{e}\pm o(1),$$
where $x$ is a random variable taking values in $\{0,1\}^n$ according to the $p$-biased distribution. Since $\Pr[x_A=(0,\ldots,0)] \ge 1 -|A| p =1-o(1)$, we conclude that for every function $g$ that depends only on the coordinates in $A$, we have $\| f-g\|_1 \ge \frac{1}{e} - o(1)$.
\end{example}

No essential progress on the case of small $p$ was made until the breakthrough work of Friedgut~\cite{MR1678031}. He gave a satisfactory description for functions with small total influences when they correspond to graph or hypergraph properties. Friedgut's theorem  is now an important tool in the study of the threshold behavior of graph properties (see~\cite{MR2116574}). A graph property on $n$-vertex graphs is modeled by a function $f:\{0,1\}^{n \choose 2} \rightarrow \{0,1\}$, where each  of the ${n \choose 2}$ coordinates correspond to the presence of one of the  ${n \choose 2}$ possible edges. Since a graph property is invariant under graph isomorphisms, every such function $f$ is invariant under several permutations of its coordinates. Various steps of Friedgut's proof rely heavily on these symmetry assumptions, and do not extend to the more general settings. However he conjectured~\cite{MR1678031} that his theorem holds without requiring  any symmetry assumptions.

In Theorem~\ref{thm:main},  we  generalize the core of Friedgut's work~\cite{MR1678031} form graph properties on their corresponding $p$-biased probability space to general properties (with no symmetry assumptions) on general product probability spaces. Previously, in this general setting,  the situation was only partially understood by a result of Bourgain:

\begin{theorem}[{\cite[Proposition 1]{Bourgain99}}]
\label{thm:Bourgain}
Let $f:(\{0,1\}^n,\mu_p^n) \rightarrow \{0,1\}$ be an increasing function with $\int f= \alpha >0$, and suppose $p=o_{I_f}(1)$.  There exist $\delta(\alpha),C(\alpha)>0$ and a subset $S \subseteq [n]$ with $|S| \le C(\alpha) I_f$ such that
$$\Ex [f(x)| x_S=(1,\ldots,1)] \ge \alpha+\delta(\alpha),$$
where $x$ is a random variables taking values in $\{0,1\}^n$ according to the $p$-biased distribution.
\end{theorem}

Let  $\mathcal{J}=\{J_S\}_{S \subseteq [n]}$ be a collection of measurable functions $J_S:X^S \rightarrow \{0,1\}$.
Define the map $J_\mathcal{J}:X^n \rightarrow \mathcal{P}([n])$ as $J_\mathcal{J}:x \mapsto \bigcup_{S \subseteq [n], J_S(x)=1} S$, and note that the map $x \mapsto (J_\mathcal{J}(x), x_{J_\mathcal{J}(x)})$ is measurable. Let $\mathcal{F}_\mathcal{J}$ be the sub-$\sigma$-algebra on $X^n$ induced by the map $x \mapsto (J_\mathcal{J}(x),x_{J_\mathcal{J}(x)})$, i.e. it is the coarsest $\sigma$-algebra on $X^n$ which makes this map measurable.

\begin{definition}
\label{def:pseudojunta}
For a constant $k>0$, a \emph{$k$-pseudo-junta} is a function $f:X^n \rightarrow \{0,1\}$ which is measurable with respect to $\mathcal{F}_\mathcal{J}$ for some $\mathcal{J}$ satisfying
$$\int |J_\mathcal{J}(x)| dx \le k.$$
\end{definition}

The functions that depend on a small number of coordinates are usually referred to as \emph{juntas}. The following example explains the choice of the name ``pseudo-junta'' in Definition~\ref{def:pseudojunta}.

\begin{example}
Consider a set $A \subseteq [n]$ with $|A| \le k$. Every measurable function $f:X^n \rightarrow \{0,1\}$ that depends only on coordinates in $A$ is a $k$-pseudo-junta. Indeed let $\mathcal{J}$ be  defined by setting $J_S \equiv 1$ if $S=A$, and $J_S \equiv 0$ otherwise. Then $J_\mathcal{J} \equiv A$, and hence $f$ is measurable with respect to $\mathcal{F}_\mathcal{J}$. Furthermore $\int |J_\mathcal{J}(x)| dx = k$.
\end{example}

The following simple proposition shows that $k$-pseudo-juntas have small total influences.

\begin{proposition}
\label{prop:direct}
Let  $f:X^n \rightarrow \{0,1\}$ be a $k$-pseudo-junta. Then $I_f \le 2k$.
\end{proposition}
\begin{proof}
Since $f$ is a $k$-pseudo-junta, it is measurable with respect to $\mathcal{F}_\mathcal{J}$ for some $\mathcal{J}$ with $\int |J_\mathcal{J}(x)| \le k$.
Let $x_1,\ldots,x_n,y_1,\ldots,y_n$ be i.i.d. random variables taking values in $X$ according to its probability measure, and define the random variables $x=(x_1,\ldots,x_n)$ and $x^{(j)}=(x_1,\ldots,x_{j-1},y_j,x_{j+1},\ldots,x_n)$ for every $j \in [n]$.  It follows from the definition
of $J_{\mathcal{J}}$ and the assumption that $f$ is measurable with respect to  $\mathcal{F}_\mathcal{J}$ that $j \in J_\mathcal{J}(x) \cup J_\mathcal{J}(x^{(j)})$ if $f(x) \neq f(x^{(j)})$. Hence
\begin{eqnarray*}
I_f &=& \sum_{j \in [n]}  \Pr\left[f(x) \neq f(x^{(j)})\right] \le \sum_{j \in [n]} \Pr\left[j \in J_\mathcal{J}(x) \cup J_\mathcal{J}(x^{(j)})\right] \\ &\le& 2 \sum_{j \in [n]} \Pr[j \in J_\mathcal{J}(x)] \le 2 \int |J_\mathcal{J}(x)| dx \le 2k.
\end{eqnarray*}
\end{proof}

Our main result is an inverse theorem which says that essentially the inverse of Proposition~\ref{prop:direct} is also true.

\begin{theorem}[Main Theorem]
\label{thm:main}
Consider a measurable function $f:X^n \rightarrow \{0,1\}$. For every $\epsilon>0$, there exists a
$e^{10^{15}  \epsilon^{-3} \lceil I_f \rceil ^3}$-pseudo-junta $h:X^n \rightarrow \{0,1\}$ such that $\| f - h\|_1 \le \epsilon$.
\end{theorem}

The key point in Theorem~\ref{thm:main} is that the bound $e^{10^{15}  \epsilon^{-3} \lceil I_f \rceil ^3}$ has no dependence on $n$.
%

\begin{example}
Consider the function $f$ in Example~\ref{ex:simple}. Define $\mathcal{J}=\{J_S\}_{S \subseteq [n]}$ as follows. For every $S \subseteq [n]$, let $J_S:\{0,1\}^S \rightarrow \{0,1\}$ be defined as
$$J_S: x\mapsto
\left\{
\begin{array}{lcl}
1 &\qquad & x_S=(1,\ldots,1), \\
0 &\qquad & \mbox{otherwise}.
\end{array}
\right.
 $$
Note that $J_\mathcal{J}: x \mapsto \{i: x_i=1\}$, and so $\int |J_\mathcal{J}(x)| dx=  p n=1$. Furthermore $\mathcal{F}_\mathcal{J}$ is the original discrete $\sigma$-algebra on $\{0,1\}^n$ and hence $f$ is measurable with respect to it. So for this example, in Theorem~\ref{thm:main} one can take $h=f$.
\end{example}


Due to the connection to the  threshold phenomenon, the case of the $p$-biased distribution in Theorem~\ref{thm:main} is of particular interest. For this probability space, the proof of Theorem~\ref{thm:main} can be significantly simplified  as many steps in the proof are to address the difficulties that arise in dealing with general probability spaces. Furthermore this simpler proof provides some improvement over the
bound $e^{10^{15}  \epsilon^{-3} \lceil I_f \rceil ^3}$ of Theorem~\ref{thm:main}. Therefore we  first state and prove the main theorem in the case of the $p$-biased measure, and then in Section~\ref{sec:general} we present the more complicated proof of the general case.

\begin{theorem}[Main Theorem, the $p$-biased case]
\label{thm:mainpbiased}
Consider a function $f:(\{0,1\}^n,\mu_p^n) \rightarrow \{0,1\}$. For every $\epsilon>0$, there exists a
$e^{10^{10}  \epsilon^{-2} \lceil I_f \rceil ^2}$-pseudo-junta $h:(\{0,1\}^n,\mu_p^n) \rightarrow \{0,1\}$ such that $\| f - h\|_1 \le \epsilon$. \end{theorem}

\begin{remark}
\label{rem:otherProps}
In Theorem~\ref{thm:mainpbiased} the function  $h$ is a pseudo-junta, and so it is measurable with respect to a $\sigma$-algebra $\mathcal{F}_\mathcal{J}$ for a certain collection $\mathcal{J}$ of  functions $J_S:\{0,1\}^S  \rightarrow \{0,1\}$.  It follows from the proof of Theorem~\ref{thm:mainpbiased} that for $p \le 1/2$, it is  possible to  assume that $J_S$ are increasing. Similarly for $p \ge 1/2$, it is possible to assume that  $J_S$ are decreasing.
\end{remark}

\subsection{Increasing set properties\label{sec:increasing}}
Let $p =o(1)$ and $f:(\{0,1\}^n,\mu_p^n) \rightarrow \{0,1\}$ be an \emph{increasing} function with $\int f=\alpha>0$.
Bourgain's result (Theorem~\ref{thm:Bourgain}) roughly speaking  says that when $I_f$ is small, it is possible to assign the value $1$ to a small set of coordinates such that restricting to this assignment increases the expected value of $f$ in a non-negligible amount. The next corollary to Theorem~\ref{thm:mainpbiased} improves this result as it shows that it is possible to increase the expected value to arbitrarily close to $1$.
\begin{corollary}
\label{cor:monotone}
Let $f:(\{0,1\}^n,\mu_p^n) \rightarrow \{0,1\}$ be an increasing function with $\alpha=\int f$, and $\epsilon>0$ be a constant. If $p \le \frac{1}{2}$, then there exists a subset $S \subseteq [n]$ with $|S| \le e^{10^{12} \lceil I_f \rceil^2 \alpha^{-2} \epsilon^{-2}}$ such that
$$\Ex [f(x)| x_S=(1,\ldots,1)] \ge 1-\epsilon,$$
where $x$ is a random variables taking values in $\{0,1\}^n$ according to the $p$-biased distribution.
\end{corollary}
\begin{proof}
Theorem~\ref{thm:mainpbiased}  shows that there exists a collection $\mathcal{J}=\{J_S\}_{S \subseteq [n]}$  with $$\int |J_\mathcal{J}(x)| \le e^{10^{11} \lceil I_f \rceil^2 \alpha^{-2} \epsilon^{-2}},$$ and a function $h:\{0,1\}^n \rightarrow \{0,1\}$, measurable with respect to $\mathcal{F}_\mathcal{J}$, such that $\| f - h\|_1 \le \frac{\alpha \epsilon}{2}$. Set $k:= e^{10^{12} \lceil I_f \rceil^2 \alpha^{-2} \epsilon^{-2}}$, and
$$ \beta := \max_{\substack{S \subseteq [n]:|S| \le k \\ y \in \{0,1\}^S}} \Pr \left[f(x)=1 | \mbox{$x_{S}=y$ and $J_\mathcal{J}(x)=S$} \right],$$
where $x$ is a random variables taking values in $\{0,1\}^n$ according to the $p$-biased distribution. Since $f$ and $h$ are Boolean and $h$ is measurable with respect to $\mathcal{F}_\mathcal{J}$, we have
\begin{eqnarray*}
\|f-h\|_1 &=& \int 1_{[f \neq h]} \ge  \int 1_{[|J_\mathcal{J}| \le k]} 1_{[f = 0]} 1_{[h = 1]} + \int 1_{[|J_\mathcal{J}| \le k]} 1_{[f = 1]} 1_{[h = 0]}  \\
& \ge& (1-\beta) \int 1_{[|J_\mathcal{J}| \le k]}  1_{[h = 1]} + \int 1_{[|J_\mathcal{J}| \le k]} 1_{[f = 1]} 1_{[h = 0]} \ge (1-\beta) \int 1_{[|J_\mathcal{J}| \le k]} 1_{[f = 1]}
\\ &\ge& (1-\beta) \left(\int f - \int 1_{[|J_\mathcal{J}| > k]}\right) \ge (1-\beta)\left( \alpha -  k^{-1} \int |J_\mathcal{J}| \right)\ge \frac{3}{4}(1-\beta) \alpha,
\end{eqnarray*}
which together with $\| f - h\|_1 \le \frac{\alpha \epsilon}{2}$ implies that $\beta \ge 1- \epsilon$.  Hence there exists a subset $S \subseteq [n]$ with $|S| \le k$ and an element $y_0 \in \{0,1\}^S$ such that
\begin{equation}
\label{eq:denseAtom}
\Ex \left[f(x) | \mbox{$x_{S}=y_0$ and $J_\mathcal{J}(x)=S$} \right]\ge 1 - \epsilon,
\end{equation}
where $x$ is a random variables taking values in $\{0,1\}^n$ according to the $p$-biased distribution.
Define $h_1, h_2:\{0,1\}^{[n]\setminus S} \rightarrow \{0,1\}$ as $h_1:x \mapsto f(y_0,x)$, and $h_2:x \mapsto 1_{[J_\mathcal{J}(y_0,x)=S]}$. We can rewrite (\ref{eq:denseAtom}) as
\begin{equation}
\label{eq:denseAtomRewrite}
\Ex \left[h_1(z) | h_2(z)=1 \right]\ge 1 - \epsilon,
\end{equation}
where $z$ is a random variable taking values in $\{0,1\}^{[n]\setminus S}$ according to the $p$-biased distribution.
Since $f$ is increasing, $h_1$ is increasing.  Furthermore by Remark~\ref{rem:otherProps} we can assume that the functions $J_S$ are increasing. Thus  $h_2$ is decreasing, and then it follows from the classical FKG inequality (see~\cite{AlonSpencer1992}) that
$$\Ex [f(x)| x_S=(1,\ldots,1)] \ge \Ex \left[f(x) | x_{S}=y_0 \right] = \Ex[h_1(z)] \ge \Ex \left[h_1(z) | h_2(z)=1 \right] \ge 1 - \epsilon.$$
\end{proof}

The following example shows that Corollary~\ref{cor:monotone} fails to be valid when the function $f$ is not necessarily increasing.
\begin{example}
\label{ex:parity}
Set $p=n^{-1}$, and let $f:\{0,1\}^n \rightarrow \{0,1\}$ be defined as $f(x)=1$ if and only if $\sum_{i=1}^n x_i \equiv 0  (\mod 2)$. Similar to Example~\ref{ex:simple}, $I_f(1)=\ldots=I_f(n) \le 2 p$, and so $I_f \le 2$. Now consider a constant size  $A \subseteq [n]$ and any $y \in \{0,1\}^A$. Let $a \in \{0,1\}$ be such that $a \equiv \sum_{i \in A} y_i (\mod 2)$. Then we have
$$\frac{1}{e}-o(1) \le (1-p)^{n-|A|} \le \Ex[f(x)=a | x_A=y] \le 1 - (n-|A|)p (1-p)^{n-|A|-1} \le 1-\frac{1}{e} +o(1),$$
where $x$ is a random variables taking values in $\{0,1\}^n$ according to the $p$-biased distribution.
So for every sufficiently large $n$, not only $f$ is far from being determined by the coordinates in $A$, but also for every $y \in \{0,1\}^A$, both $\Ex[f(x)=0 | x_A=y]$ and $\Ex[f(x)=1 | x_A=y]$ are well separated from  $0$ and $1$ as they both belong to the interval $[\frac{1}{2e},1 - \frac{1}{2e}]$.
\end{example}

\section{Generalized Walsh expansion \label{sec:prel}}
In this short section we review some basic facts about the generalized Walsh expansions which are first defined by Hoeffding in~\cite{MR0026294} (See also~\cite{MR615434}).
Let $L_2(X^n)$ denote the set of functions $f:X^n \rightarrow \mathbb{C}$ that satisfy $\int |f(x)|^2 dx < \infty$. Consider a subset $S \subseteq [n]$, and a function $f \in L_2(X^n)$. Then $\int f(x) dx_S := \int f(x) \prod_{i \in S} dx_i$ denotes the integral with respect to the coordinates in $S$.

\begin{definition}
\label{def:Walsh}
The \emph{generalized Walsh expansion} of a  function $f \in L_2(X^n)$ is the unique expansion $f= \sum_{S \subseteq [n]} F_S$ that satisfies the following two properties:
\begin{itemize}
\item[{\bf (i)}] For every $S \subseteq [n]$, the function $F_S$ depends only on the coordinates in $S$, i.e. $F_S(x)=F_S(x_S)$;
\item[{\bf (ii)}] $\int F_S(x) dx_i \equiv 0$, for every $S \subseteq [n]$ and every $i \in S$.
\end{itemize}
\end{definition}

Note that it follows from  Definition~\ref{def:Walsh} (i) and (ii) that for every $T \subseteq [n]$, we have
$\int f dx_{[n] \setminus T} = \sum_{S \subseteq T} F_S$. Consequently for every $y \in X^n$,
\begin{equation}
\label{eq:Walsh}
F_S(y) = \sum_{T \subseteq S} (-1)^{|S \setminus T|} \int f(y_T,x_{[n] \setminus T}) dx_{[n] \setminus T},
\end{equation}
which shows that the generalized Walsh expansion is unique. It follows from (\ref{eq:Walsh}) that for every $S \subseteq [n]$, we have
\begin{equation}
\label{eq:FSinfinity}
\|F_S\|_\infty \le 2^{|S|} \|f\|_\infty.
\end{equation}

Consider two  subsets $S_1,S_2 \subseteq [n]$. If $S_1 \neq S_2$, then  Definition~\ref{def:Walsh} (i) and (ii) guarantee that $\int F_{S_1}  \overline{F_{S_2}} = 0$, or in other words the functions $\{F_S : S \subseteq [n]\}$ are pairwise orthogonal. As a consequence we have Parseval's identity:
\begin{equation}
\label{eq:Parseval}
\|f\|_2^2 = \sum_{S \subseteq [n]} \|F_S\|_2^2.
\end{equation}

The influences of variables have simple descriptions in terms of the generalized Walsh expansion. For a measurable function $f:X^n \rightarrow \{0,1\}$, and every $i \in [n]$, define
$$f^{(i)} = f - \int f dx_i = \sum_{S: i \in S} F_S.$$
It is easy to see that $I_f(i) = 2\| f^{(i)}\|_2^2,$ which by Parseval's identity implies
$$I_f(i) = 2 \sum_{S: i \in S} \|F_S\|_2^2.$$
Thus the total influence of $f$ is given by the formula
\begin{equation}
\label{eq:InfInWalsh}
I_f = 2\sum_{S \subseteq [n]} |S| \|F_S\|_2^2.
\end{equation}

\section{The $p$-biased case: Proof of Theorem~\ref{thm:mainpbiased}}
Without loss of generality we assume $p \le \frac{1}{2}$.  We start from the generalized Walsh expansion  $f = \sum_{S \subseteq [n]} F_S$. In the first two steps we prone this expansion  by removing some insignificant terms from it. We will arrive at a set $\mathcal{S} \subseteq \mathcal{P}([n])$ such that $\|f - \sum_{S \in \mathcal{S}} F_S\|_2$ is small and meanwhile the functions $F_S$ with $S \in \mathcal{S}$ satisfy certain properties. Then in the main step of the proof we define the collection $\mathcal{J}=\{J_S\}_{S \subseteq [n]}$, and show that $\left\|g - \Ex[g|\mathcal{F}_\mathcal{J}] \right\|_2$ is small for $g:=\sum_{S \in \mathcal{S}} F_S$. Once this is established, it is straightforward to finish the proof.  In the case  of the $p$-biased distribution, the functions $F_S$ have simple descriptions: There exists real constants $\{\widehat{f}(S)\}_{S \subseteq [n]}$ such that $F_S(x) = \widehat{f}(S) \prod_{i \in S} r(x_i)$ where $r(0)=-\sqrt{\frac{p}{1-p}}$ and $r(1)=\sqrt{\frac{1-p}{p}}$.  Note that $\widehat{f}(S) = \int f(x) \prod_{i \in S} r(x_i) dx$ which implies $|\widehat{f}(S)| \le p^{|S|/2}$. These properties of $F_S$ are crucial in this proof.

We abbreviate $\mu_p^n$ to $\mu$. We will define various constants which for the convenience of the reader are listed here:
$$
\begin{array}{lclclclclcl}
C &:=& \lceil I_f \rceil, &\qquad & \epsilon_0 &:=& 10^{-3} \epsilon, &\qquad& k &:=&  C \epsilon_0^{-1} = 10^{3} \epsilon^{-1} C, \\
\delta &:=& 2^{-10^2k^2}, && \epsilon_1 &:=&  3^{-10 k^2} \epsilon_0^{10 k}. &&  &&    \\
\end{array}
$$

Set $C$, $k$, $\epsilon_0$  as above, and let $\epsilon_1>0$ be a constant to be determined later. Let
$$\mathcal{S}:= \left\{S \subseteq [n]: |S| \le k, \|F_S\|_\infty > \epsilon_1 \right\}.$$
We wish to show that  the total contribution of $F_S$ with $S \not\in \mathcal{S}$ to $\|f\|_2^2$ is insignificant. First we deal with $S$ with $|S| >k$.

\subsection{Step I: High frequencies}
By (\ref{eq:Parseval}) and (\ref{eq:InfInWalsh}) we have
\begin{equation}
\label{eq:pbiased:cutFourier}
\sum_{S: |S| \ge k} \|F_S\|_2^2 \le \frac{I_f}{k} \le \epsilon_0.
\end{equation}
\subsection{Step II: Bourgain's argument}
Next we deal with $S \not\in \mathcal{S}$ with $|S| \le k$.
This step is not self-contained and is based on Bourgain's argument in~\cite{Bourgain99}. Following his proof (but substituting $k$ in place of $10C$), one can replace~\cite[Inequality (2.17)]{Bourgain99} with the following:
\begin{eqnarray*}
\int \sum_{S:|S| \le k} F_S^2  1_{[|F_S| \le \epsilon_1]} &\le& M^k \int \max_{S: |S| \le k} F_S^2  1_{[|F_S| \le \epsilon_1]} + 2 C 3^{k/2} \varepsilon^{2/3} + C^{1/2} \frac{3^k}{M^{1/2} \varepsilon}
\\ &\le& M^k \epsilon_1^2 + 2C 3^{k/2} \varepsilon^{2/3} + C^{1/2} \frac{3^k}{M^{1/2} \varepsilon},
\end{eqnarray*}
where $\epsilon_1, \varepsilon, M >0$  are arbitrary constants.
Then by setting $\varepsilon := 3^{-3k} \epsilon_0^{3}$, $M:= 3^{10 k} \epsilon_0^{-10}$, and $\epsilon_1 :=  3^{-10 k^2} \epsilon_0^{10 k}$, we obtain
\begin{equation}
\label{eq:pbiased:Bourgain}
\sum_{S: |S| \le k, S \not\in \mathcal{S}} \|F_S\|_2^2 \le \int \sum_{S:|S| \le k} F_S^2  1_{[|F_S| \le \epsilon_1]} \le \frac{\epsilon_0}{k}.
\end{equation}
\subsection{Step III: Main Step}
Next we define the collection of functions $\mathcal{J}=\{J_T\}_{T \subseteq [n]}$.
Set $\delta:=2^{-10^2k^2}$ and for every $T \subseteq [n]$ with $|T| \le k$, define $J_T: \{0,1\}^T \rightarrow \{0,1\}$ as
$$
J_T : y \mapsto
\left\{
\begin{array}{lcl}
1 &\qquad & \sum_{S \in \mathcal{S}, S \supseteq T} \int 1_{[|F_S| \ge \epsilon_1]} \ge \delta \mu(y),
\\ 0 & & \mbox{otherwise}.
\end{array}
\right.
$$
For $T \subseteq [n]$ with $|T| > k$, define $J_T \equiv 0$. Note that since $p \le \frac{1}{2}$, the functions $J_T$ are increasing.
The required bound on $\int |J_{\mathcal{J}}(x)|$ can be verified easily:
\begin{eqnarray*}
\int |J_\mathcal{J}| &\le& \sum_{T \subseteq [n], |T| \le k} |T| \int J_T \le k \sum_{T:|T| \le k} \int J_T \le \frac{k}{\delta}  \sum_{T:|T| \le k} \sum_{S \in \mathcal{S}, S \supseteq T} \int 1_{[|F_S| \ge \epsilon_1]}  \int \frac{1}{\mu(x_T)}  \\ &=&  \frac{k}{\delta} \sum_{T:|T| \le k} \sum_{S \in \mathcal{S}, S \supseteq T} 2^{|T|}\int 1_{[|F_S| \ge \epsilon_1]} \le \frac{k 2^{2k}}{\delta}    \sum_{S \in \mathcal{S}} \frac{\|F_S\|_2^2}{\epsilon_1^2}  \le \frac{k 2^{2k}}{\delta \epsilon_1^2} \le
2^{10^3k^2} \le e^{10^{10} \epsilon^{-2} \lceil I_f \rceil^2}.
\end{eqnarray*}

By  (\ref{eq:pbiased:cutFourier}) and (\ref{eq:pbiased:Bourgain}) for  $g :=  \sum_{S \in \mathcal{S}} F_S$, we have $\|f-g\|_2^2 \le  2\epsilon_0 $. Our goal is now to show that $\left\|g - \Ex[g|\mathcal{F}_\mathcal{J}]\right\|_2^2$ is small. Note that $\Ex[g|\mathcal{F}_\mathcal{J}] = \sum_{S \in \mathcal{S}} \Ex[F_S | \mathcal{F}_\mathcal{J}]$.  However since $\mathcal{F}_{\mathcal{J}}$ depends on all coordinates, it is difficult to analyze $\Ex[F_S | \mathcal{F}_{\mathcal{J}}]$ directly. To remedy this we define some auxiliary $\sigma$-algebras.
For every $S \subseteq [n]$, define the collection $\mathcal{J}_S := \{\tilde{J}_T\}_{T \subseteq [n]}$ in the following way. For every $T \subseteq [n]$, we set $\tilde{J}_T := J_T$ if  $T \subseteq S$, and  $\tilde{J}_T \equiv 0$ otherwise. Define $\tilde{g} := \sum_{S \in \mathcal{S}} \tilde{F}_S$, where
$\tilde{F}_S:=\Ex[F_S  | \mathcal{F}_{\mathcal{J}_S}]$.  For every $S \subseteq [n]$,
the function $\tilde{F}_S$ depends only on the coordinates in $S$ and furthermore
the $\sigma$-algebra $\mathcal{F}_{\mathcal{J}_S}$ is coarser than the $\sigma$-algebra
$\mathcal{F}_{\mathcal{J}}$. It follows from the latter that $\tilde{g}$ is measurable with respect to $\mathcal{F}_{\mathcal{J}}$ and hence
\begin{equation}
\label{eq:BoundConditional1}
\left\|g - \Ex[g | \mathcal{F}_{\mathcal{J}}] \right\|_2 \le \left\|g - \tilde{g}\right\|_2.
\end{equation}
So in order to bound $\left\|g - \Ex[g | \mathcal{F}_{\mathcal{J}}] \right\|_2$, it suffices to bound  $\|g - \tilde{g}\|_2$. Trivially for every $S \in \mathcal{S}$,
$$\int F_S \tilde{F}_S =\int \tilde{F}_S^2.$$
For  $S_1, S_2 \in \mathcal{S}$, since $\tilde{F}_{S_1}$ and $\tilde{F}_{S_2}$  depend respectively only on the coordinates in $S_1$ and $S_2$, if $S_1 \not\subseteq S_2$ and $S_2 \not\subseteq S_1$, then by Definition~\ref{def:Walsh} (ii),
$$\int F_{S_1}F_{S_2} = \int F_{S_1} \tilde{F}_{S_2} =\int \tilde{F}_{S_1} F_{S_2} = 0.$$
Similarly if $S_1 \cap S_2 = \emptyset$, then
$$
\int\tilde{F}_{S_1} \tilde{F}_{S_2} = \left(\int \tilde{F}_{S_1} \right)\left(\int \tilde{F}_{S_2} \right) =\left(\int F_{S_1}\right) \left(\int F_{S_2}\right)=0.
$$
Thus
\begin{eqnarray}
\nonumber
\|g - \tilde{g}\|_2^2 &=& \int \left(\sum_{S \in \mathcal{S}} F_S - \tilde{F}_S \right)^2 \le (k+1)\sum_{r=0}^k \int \left(\sum_{S \in \mathcal{S}, |S|=r} F_S - \tilde{F}_S \right)^2\\
&=& \nonumber  (k+1)\int \sum_{S \in \mathcal{S}} F_S^2 -  \tilde{F}_S^2 +  \sum_{\substack{S_1,S_2 \in \mathcal{S},  S_1 \cap S_2 \neq \emptyset \\ S_1 \neq S_2, |S_1|=|S_2|}}\tilde{F}_{S_1} \tilde{F}_{S_2} \\
&\le& 2k\int \sum_{S \in \mathcal{S}} F_S^2 - \tilde{F}_S^2 \label{eq:same1}
\\
&&+ 2k \sum_{\substack{S_1,S_2 \in \mathcal{S}, S_1 \neq S_2\\ S_1 \cap S_2 \neq \emptyset}}  \left|\int\tilde{F}_{S_1} \tilde{F}_{S_2}\right|. \label{eq:distinct1}
\end{eqnarray}

\noindent \emph{Bounding (\ref{eq:same1}):} Note that if $|F_S(x)| \ge \epsilon_1$, then $J_S(x)=1$ and hence $F_S(x)=\tilde{F}_S(x)$. Consequently by (\ref{eq:pbiased:Bourgain}), we have

\begin{eqnarray}
\nonumber
(\ref{eq:same1})&=&2k \int \sum_{S \in \mathcal{S}} |F_S -  \tilde{F}_S|^2  \le  2k \int \sum_{S \in \mathcal{S}} |F_S|^2 1_{[|F_S| \le \epsilon_1]}
\le  2k \frac{\epsilon_0}{k} \le 2 \epsilon_0.
\label{eq:boundSame1}
\end{eqnarray}

\noindent \emph{Bounding (\ref{eq:distinct1}):}
First let us prove a simple inequality.
Consider a set $S \subseteq [n]$, and note that for every $x \in \{0,1\}^S$,
\begin{equation*}
\mu(x) |F_S(x)|= (p(1-p))^{|S|/2}|\widehat{f}(S)|
\end{equation*}
is a constant that does not depend on $x$. This shows that $|\tilde{F}_S(x)| \le 2^{|S|} |F_S(x)|$ for every $x \in \{0,1\}^S$. It follows that for every  subset $T \subseteq S$, and every $y \in \{0,1\}^T$, we have
\begin{equation}
\label{eq:FSofY}
\left|\int \tilde{F}_{S}(y,x_{S \setminus T}) dx_{S \setminus T}\right| \le 2^{|S|}   \int |F_{S}(y,x_{S \setminus T})| dx_{S \setminus T} = 2^{|S|+|S \setminus T|} \frac{(p(1-p))^{|S|/2}|\widehat{f}(S)|}{\mu(y)} \le 2^{2|S|} \frac{p^{|S|}}{\mu(y)}.
\end{equation}
Moreover if $T \subsetneq S$ and $J_T(y)=1$, then
\begin{equation}
\label{eq:JOneZero}
\int \tilde{F}_{S}(y,x_{S \setminus T}) dx_{S \setminus T} = \int F_{S}(y,x_{S \setminus T}) dx_{S \setminus T} = 0.
\end{equation}
Consider \emph{distinct} $S_1,S_2 \in \mathcal{S}$ and $y \in \{0,1\}^T$ where $T:=S_1 \cap S_2$. If $J_T(y)=1$, then by (\ref{eq:JOneZero}),
$$\int  \tilde{F}_{S_1}(y,x_{[n] \setminus T}) \tilde{F}_{S_2}(y,x_{[n] \setminus T}) dx_{[n] \setminus T}=0.$$
Hence by (\ref{eq:FSofY}), we have
\begin{eqnarray*}
\left|\int  \tilde{F}_{S_1} \tilde{F}_{S_2}\right| &=& \left|\int  \tilde{F}_{S_1}(x_{T},x_{S_1 \setminus T})
\tilde{F}_{S_2}(x_{T},x_{S_2 \setminus T}) 1_{[J_T(x_T)=0]}\right| \\ &=& \left|\int \left(\int \tilde{F}_{S_1}(x_{T},x_{S_1 \setminus T})
 dx_{S_1 \setminus T}\right)\left(\int \tilde{F}_{S_2}(x_{T},x_{S_2 \setminus T}) dx_{S_2 \setminus T}\right) 1_{[J_T(x_T)=0]} dx_T \right|\\
&\le& 2^{2|S_1|+2|S_2|} p^{|S_1|+|S_2|} \int \frac{1_{[J_T(x_T)=0]}}{\mu(x_T)^2} dx_T.
\end{eqnarray*}
Since for $S \in \mathcal{S}$, we have $\int 1_{[|F_S| \ge \epsilon_1]}  \ge p^{|S|}$, it follows from the definition of $J_T$ that for every $T \subseteq [n]$,
$$\sum_{S \in \mathcal{S}, S \supseteq T} p^{|S|} \int \frac{1_{[J_T(x_T)=0]}}{\mu(x_T)^2} dx_T =
\sum_{S \in \mathcal{S}, S \supseteq T} p^{|S|} \sum_{x_T \in \{0,1\}^T} \frac{1_{[J_T(x_T)=0]}}{\mu(x_T)}  \le\sum_{x_T \in \{0,1\}^T} \delta \le  2^{|T|} \delta,$$
and thus
\begin{eqnarray*}
(\ref{eq:distinct1}) &\le&  2k \sum_{T \subseteq [n]}\sum_{\substack{S_1,S_2 \in \mathcal{S}\\ S_1 \neq S_2, S_1 \cap S_2 =T}}  2^{4k}p^{|S_1|+|S_2|} \int \frac{1_{[J_T(x_T)=0]}}{\mu(x_T)^2} \le  2^{5k} \sum_{T \subseteq [n]} \sum_{S \in \mathcal{S}, S \supseteq T} p^{|S|} 2^{|T|} \delta \\
&\le& 2^{6k}\delta \sum_{T \subseteq [n]}  \sum_{S \in \mathcal{S}, S \supseteq T} p^{|S|} \le 2^{7k} \delta \sum_{S \in \mathcal{S}} p^{|S|} \le  2^{7k} \delta \sum_{S \in \mathcal{S}} \frac{\|F_S\|_2^2}{\epsilon_1^2} \le  2^{7k} \delta \epsilon_1^{-2} \le \epsilon_0.
\end{eqnarray*}
Now we conclude from our bounds on (\ref{eq:same1}) and (\ref{eq:distinct1}) that $\|g - \tilde{g}\|_2^2 \le 3\epsilon_0$.
\subsection{Step IV: Finishing the proof}
In the previous steps we have shown that both $\|f -g\|_2^2$  and  $\|g - \tilde{g}\|_2^2$ are  small. It follows
$$\|f - \Ex[f|\mathcal{F}_\mathcal{J}]\|_2^2 \le \|f - \Ex[g|\mathcal{F}_\mathcal{J}]\|_2^2 \le
2\|f -g\|_2^2 + 2\|g- \Ex[g|\mathcal{F}_\mathcal{J}]\|_2^2 \le 4 \epsilon_0 + 2\|g-\tilde{g}\|_2^2 \le 10 \epsilon_0.$$
Define $h:\{0,1\}^n \rightarrow \{0,1\}$ as
$$h(x) :=
\left\{
\begin{array}{lcl}
0 & \qquad & \Ex[f|\mathcal{F}_\mathcal{J}]  \le \frac{1}{2}\\
1 & & \Ex[f|\mathcal{F}_\mathcal{J}]  > \frac{1}{2}.
\end{array}
\right.
$$
Note that $h$ is a Boolean function and it is measurable with respect to $\mathcal{F}_\mathcal{J}$. Since $f$ is a Boolean function, we have
\begin{eqnarray*}
\|f-h\|_1  &=& \int |f-h|^2\le \int 4|f-\Ex[f|\mathcal{F}_\mathcal{J}]|^2 \le 40 \epsilon_0  \le \epsilon.
\end{eqnarray*}

\section{The general case: Proof of Theorem~\ref{thm:main} \label{sec:general}}
First note that by Remark~\ref{rem:finiteness} we can assume that $X$ is a finite probability space. We start from  the generalized Walsh expansion $f = \sum_{S \subseteq [n]} F_S$. The first two steps are similar to the proof of Theorem~\ref{thm:mainpbiased} where we prone this expansion by removing some insignificant terms from it. We will arrive at a set $\mathcal{S} \subseteq \mathcal{P}([n])$ such that $\|f - \sum_{S \in \mathcal{S}} F_S\|_2$ is small and meanwhile the functions $F_S$ with $S \in \mathcal{S}$ satisfy certain properties. In general the functions $F_S$ are not necessarily as well-behaved as in the case of the $p$-biased distribution. So in the third step we approximate $\sum_{S \in \mathcal{S}} F_S$ with $g:=\sum_{S \in \mathcal{S}} G_S$ where for every $S \in \mathcal{S}$, the function $G_S$ (similar to $F_S$) satisfies the conditions in Definition~\ref{def:Walsh} (i) and (ii) and furthermore it has some other desirable properties. In the fourth and fifth steps we define the collection $\mathcal{J}=\{J_S\}_{S \subseteq [n]}$, and show that $\left\|g - \Ex[g|\mathcal{F}_\mathcal{J}] \right\|_2$ is small. Then in the last step which is similar to the last step of the proof of Theorem~\ref{thm:mainpbiased}, we conclude the theorem.

Set $C:=\lceil I_f \rceil$. In the proof we will define various constants which for the convenience of the reader are listed here:
$$
\begin{array}{lclclcl}
\epsilon_0 &=& 10^{-3} \epsilon, &\qquad & k &=&  C \epsilon_0^{-1} = 10^{3} \epsilon^{-1} C, \\
\delta_0 &=& 2^{-2k}, & & \epsilon_1 &=&  3^{-10 k^2} \epsilon_0^{10 k} \ge e^{-10^{8} C^2 \epsilon^{-2}}, \\
\delta &=& 2^{-10^3 k^2} \epsilon_1^{10} \ge e^{-10^{10} C^2 \epsilon^{-2}}, & & \epsilon_2 &=& \delta^{10k} \ge e^{-10^{14} C^3 \epsilon^{-3}}.
\end{array}
$$

\subsection{Step I: High frequencies}

Set $\epsilon_0:= 10^{-3} \epsilon$ and $k:=C \epsilon_0^{-1}$ and notice that by (\ref{eq:Parseval}) and (\ref{eq:InfInWalsh}) we have
\begin{equation}
\label{eq:cutFourier}
\sum_{S: |S| \ge k} \|F_S\|_2^2 \le \frac{I_f}{k} \le \epsilon_0.
\end{equation}

\subsection{Step II: Bourgain's argument}
As we mentioned in the proof of Theorem~\ref{thm:mainpbiased}, following Bourgain's proof in \cite{Bourgain99}   (but substituting $k$ in place of $10C$), one can replace~\cite[Inequality (2.17)]{Bourgain99} with the following:
\begin{eqnarray*}
\int \sum_{S:|S| \le k} F_S^2  1_{[|F_S| \le \epsilon_1]} &\le& M^k \int \max_{S: |S| \le k} F_S^2  1_{[|F_S| \le \epsilon_1]} + 2 C 3^{k/2} \varepsilon^{2/3} + C^{1/2} \frac{3^k}{M^{1/2} \varepsilon}
\\ &\le& M^k \epsilon_1^2 + 2C 3^{k/2} \varepsilon^{2/3} + C^{1/2} \frac{3^k}{M^{1/2} \varepsilon},
\end{eqnarray*}
where $\epsilon_1, \varepsilon, M >0$  are arbitrary constants. Then by setting $\varepsilon := 3^{-3k} \epsilon_0^{3}$, $M:= 3^{10 k} \epsilon_0^{-10}$, and $\epsilon_1 :=  3^{-10 k^2} \epsilon_0^{10 k}$, we obtain
\begin{equation}
\label{eq:Bourgain}
\int \sum_{S:|S| \le k} F_S^2  1_{[|F_S| \le \epsilon_1]} \le \epsilon_0^2 k^{-1}.
\end{equation}
Define
$$\mathcal{S} := \left\{S \subseteq [n]: |S| \le k, \int F_S^2 1_{[|F_S| \le \epsilon_1]} \le \epsilon_0 k^{-1} \int F_S^2 \right\},$$
and notice that by (\ref{eq:Bourgain}) we have
\begin{equation}
\label{eq:proningSmall}
\sum_{S: |S| \le k, S \not\in \mathcal{S}} \int F_S^2   \le  \sum_{S: |S| \le k} \frac{k}{\epsilon_0} \int F_S^2 1_{[|F_S| \le \epsilon_1]}
\le \epsilon_0.
\end{equation}
\subsection{Step III: Modifying the generalized Walsh functions}
Now we will focus our attention to $S \in \mathcal{S}$ as (\ref{eq:cutFourier}) and (\ref{eq:proningSmall}) show that the generalized Walsh functions $F_S$ with $S \not\in \mathcal{S}$ have a negligible contribution to the $L_2$ norm of $f$.

Define $\delta:=2^{-10^3 k^2} \epsilon_1^{10}$, and for every $S \in \mathcal{S}$, let $\psi_S:X^S \rightarrow \{0,1\}$  be defined as
$$\psi_S(y) := \left\{
\begin{array}{lcl}
1 & \qquad &\max_{T \subseteq S} \delta^{-2|S \setminus T|}\int 1_{[|F_S(y_T,x_{S \setminus T})| > \epsilon_1]} dx_{S \setminus T} \ge 1,\\
0 & & \mbox{otherwise.}
\end{array}
\right.
$$
The next lemma provides some information about the generalized Walsh expansion of $F_S \psi_S$.
\begin{lemma}
\label{lem:smooth}
Consider an $S \in \mathcal{S}$. Let $F_S \psi_S = \sum_{T \subseteq S} H_T$ be the generalized Walsh expansion of $F_S \psi_S$. Then $\|H_T\|_\infty \le \delta$, for every $T \subsetneq S$.
\end{lemma}
\begin{proof}
We claim that for every $R \subsetneq S$,
\begin{equation}
\label{eq:claim}
\left\|\int F_S \psi_S dx_{S \setminus R} \right\|_\infty \le  2^{3k} \delta^2.
\end{equation}
Suppose to the contrary that there exists  $y \in X^R$, such that $\left|\int F_S(y,x_{S \setminus R}) \psi_S(y,x_{S \setminus R}) dx_{S \setminus R} \right| >  2^{3k} \delta^2.$
Then for every $T \subseteq R$, we have
\begin{equation}
\label{eq:contra}
\delta^{-2|S \setminus T|}\int 1_{[|F_S(y_T,x_{S \setminus T})| > \epsilon_1]} dx_{S \setminus T} < 1,
\end{equation}
as otherwise $\psi_S(z)=1$, for every $z \in  X^S$ with $z_R=y$, and in this case, since $R \neq S$, we would have
$$\int F_S(y,x_{S \setminus R}) \psi_S(y,x_{S \setminus R}) dx_{S \setminus R}= \int F_S(y,x_{S \setminus R}) dx_{S \setminus R} =0.$$
Now by (\ref{eq:FSinfinity}), (\ref{eq:contra}) and the definition of $\psi_S$, we  have
\begin{eqnarray*}
\left|\int F_S(y,x_{S \setminus R}) \psi_S(y,x_{S \setminus R}) dx_{S \setminus R} \right| &\le& 2^k \int \psi_S(y,x_{S \setminus R}) dx_{S \setminus R} \\ &\le&
2^k \int \sum_{T \subseteq S,  T \not\subseteq R} \delta^{-2|S \setminus T|}\left( \int 1_{[|F_S| > \epsilon_1]} dx_{S \setminus T}\right) dx_{S \setminus R} \\
& = &  2^{k} \sum_{T \subseteq S, T \not\subseteq R}  \delta^{-2|S \setminus T|}\int 1_{[|F_S| > \epsilon_1]} dx_{S \setminus (T \cap R)} \\ &\le& 2^{2k} \delta^2 \sum_{T \subseteq R} \delta^{-2|S \setminus T|} \int 1_{[|F_S| > \epsilon_1]} dx_{S \setminus T} \le 2^{3k} \delta^2,
\end{eqnarray*}
which establishes our claim (\ref{eq:claim}). Now by (\ref{eq:Walsh})  and (\ref{eq:claim}), for every $T \subsetneq S$,  we have
\begin{eqnarray*}
\|H_T\|_\infty &=& \left\| \sum_{R \subseteq T} (-1)^{|T \setminus R|} \int F_S \psi_S dx_{[n] \setminus R} \right\|_\infty \le \sum_{R \subseteq T} \left\| \int F_S \psi_S dx_{[n] \setminus R} \right\|_\infty \le 2^{4k} \delta^2 \le \delta.
\end{eqnarray*}
\end{proof}

For every $S \in \mathcal{S}$, define $G_S := H_S$, where  $F_S \psi_S = \sum_{T \subseteq S} H_T$ is the generalized Walsh expansion of $F_S \psi_S$.  The function $G_S$ satisfies the conditions in Definition~\ref{def:Walsh} (i) and (ii). Lemma~\ref{lem:Properties} below lists some other properties of $G_S$. Part (a) shows that for $S \in \mathcal{S}$, similar to $F_S$, most of the $L_2$ weight of $G_S$ is concentrated on $\{x: |F_S(x)| \ge \epsilon_1\}$. Part (b) shows that $G_S$ is a good approximation of $F_S$ in the $L_2$ norm. Parts (c) and (d) in particular show that it is possible to bound the sum of the $L_1$ norms of $G_S$, for $S \in \mathcal{S}$. The functions $F_S$ do not necessarily satisfy this latter property, and it is for this reason that we replace them with the functions $G_S$. For every $S \in \mathcal{S}$, define $a_S:X^S \rightarrow \mathbb{R}$ as
$$a_S: y \mapsto  2^{3k} \delta^{-2k} \sum_{T \subseteq S} \int 1_{[|F_S(y_{S\setminus T},x_T)| > \epsilon_1]} dx_{T}.$$
\begin{lemma}
\label{lem:Properties}
For every $S \in \mathcal{S}$, we have
\begin{itemize}
\item[{\bf (a):}] $\int G_S^2 1_{[|F_S| \le \epsilon_1]} \le  2 \epsilon_0 k^{-1} \|F_S\|_2^2$.
\item[{\bf (b):}] $\|F_S - G_S\|_2^2 \le 10 \epsilon_0 k^{-1} \|F_S\|_2^2$.
\item[{\bf (c):}] For every $y \in X^S$, we have $|G_S(y)| \le  a_S(y)$.
\item[{\bf (d):}] $\sum_{S \in \mathcal{S}} \int a_S dx_{[n]} \le \delta^{-3k}$.
\end{itemize}
\end{lemma}
\begin{proof}
Consider $S \in \mathcal{S}$, and let $F_S \psi_S = \sum_{T \subseteq S} H_T$ be the generalized Walsh expansion of $F_S \psi_S$.
By Lemma~\ref{lem:smooth} we have
\begin{equation}
\label{eq:FSGSdiff}
\|F_S - G_S\|_\infty \le \sum_{T \subsetneq S} \|H_T\|_\infty \le 2^k \delta.
\end{equation}
Hence by (\ref{eq:FSinfinity}), (\ref{eq:Parseval}) and the assumptions that $S \in \mathcal{S}$ and $\delta=2^{-10^3 k^2} \epsilon_1^{10}$, we have
\begin{eqnarray*}
\int F_S^2 &\ge& \int F_S^2 \psi_S^2 \ge \int G_S^2 \ge \int G_S^2 1_{[|F_S| > \epsilon_1]} \ge \int (|F_S| - 2^{k} \delta)^2 1_{[|F_S| > \epsilon_1]}
\\  &\ge& \int F_S^2 1_{[|F_S| > \epsilon_1]} - 2^{k+1} \delta  \int |F_S| 1_{[|F_S| > \epsilon_1]}
\ge \int F_S^2 1_{[|F_S| > \epsilon_1]} - 2^{3k} \delta  \int  1_{[|F_S| > \epsilon_1]}
\\ &\ge&  (1- \epsilon_0 k^{-1}) \int F_S^2 - 2^{3k} \delta \epsilon_1^{-2}\int F_S^2 \ge (1-2 \epsilon_0 k^{-1}) \int F_S^2.
\end{eqnarray*}
Then
$$\int G_S^2 1_{[|F_S| \le \epsilon_1]} = \int G_S^2 - \int G_S^2 1_{[|F_S| > \epsilon_1]} \le \int F_S^2 - (1-2 \epsilon_0 k^{-1}) \int F_S^2
\le 2 \epsilon_0 k^{-1} \|F_S\|_2^2,$$
which verifies Part (a). In order to prove Part (b) note that by Part (a), (\ref{eq:FSGSdiff}) and the assumptions that $S \in \mathcal{S}$ and $\delta=2^{-10^3 k^2} \epsilon_1^{10}$, we have
\begin{eqnarray*}
\|F_S - G_S\|_2^2 &=& \int |F_S -G_S|^2 1_{[|F_S| \ge \epsilon_1]} +  \int |F_S -G_S|^2 1_{[|F_S| < \epsilon_1]}
\\ &=& 2^{2k}\delta^2 \int 1_{[|F_S| \ge \epsilon_1]} +  2 \int (F_S^2 +G_S^2) 1_{[|F_S| < \epsilon_1]}  \\
& \le & 2^{2k}\delta^2 \epsilon_1^{-2}  \int F_S^2 + 2 \int F_S^2  1_{[|F_S| < \epsilon_1]} + 2 \int G_S^2  1_{[|F_S| < \epsilon_1]}
\\ & \le& 2^{2k}\delta^2 \epsilon_1^{-2} \|F_S\|_2^2 + 2\epsilon_0 k^{-1} \|F_S\|_2^2 + 4\epsilon_0 k^{-1} \|F_S\|_2^2 \le 10 \epsilon_0 k^{-1} \|F_S\|_2^2.
\end{eqnarray*}
To prove Part (c) notice that by (\ref{eq:Walsh}), (\ref{eq:FSinfinity}), and the definition of $\psi_S$,
\begin{eqnarray*}
|G_S| &=& \left| \sum_{R \subseteq S} (-1)^{|S \setminus R|} \int F_S \psi_S dx_{S \setminus R} \right| \le \sum_{R \subseteq S} 2^k \int \psi_S dx_{S \setminus R}  \\ & \le & 2^k \sum_{R \subseteq S} \int \delta^{-2k}\sum_{T \subseteq S} \left( \int 1_{[|F_S| > \epsilon_1]} dx_{S \setminus T}\right) dx_{S \setminus R}
\\ &\le& 2^k \delta^{-2k} \sum_{R \subseteq S} \sum_{T \subseteq S} \int 1_{[|F_S| > \epsilon_1]} dx_{S \setminus (T \cap R)} \le  2^{3k} \delta^{-2k} \sum_{T \subseteq S} \int 1_{[|F_S| > \epsilon_1]} dx_{T}.
\end{eqnarray*}
It remains to prove Part (d). We have
\begin{eqnarray*}
\sum_{S \in \mathcal{S}} \int a_S dx_{[n]} &=&  2^{3k} \delta^{-2k} \sum_{S \in \mathcal{S}} \sum_{T \subseteq S} \int 1_{[|F_S| > \epsilon_1]} dx_T dx_{[n]}\le 2^{4k} \delta^{-2k}\sum_{S \in \mathcal{S}} \int 1_{[|F_S| > \epsilon_1]} \\ &=&  2^{4k} \delta^{-2k} \epsilon_1^{-2} \sum_{S \in \mathcal{S}} \|F_S\|_2^2 \le 2^{4k} \delta^{-2k} \epsilon_1^{-2} \le \delta^{-3k}.
\end{eqnarray*}
\end{proof}

\subsection{Step IV: The sigma algebra\label{sec:sigmaAlg}}
In Steps 1-3 of the proof we approximated $f$ in the $L_2$ norm with $g:=\sum_{S \in \mathcal{S}} G_S$, where the functions $G_S$ satisfy certain properties. Next we will define the collection $\mathcal{J} = \{J_S\}_{S \subseteq [n]}$ so that  $\left\|g - \Ex[g | \mathcal{F}_{\mathcal{J}}] \right\|_2$ is small, and $\int |J_\mathcal{J}|$ is bounded by a constant that does not depend on $n$. In order to define $\mathcal{J}$, we need to introduce the auxiliary functions $\xi_T: X^T \rightarrow \{0,1\}$, for every $T \subseteq [n]$ with $|T| \le k$. Set $\epsilon_2 := \delta^{10k}$ and for every $T \subseteq [n]$ with $|T| \le k$, define $\xi_T: X^T \rightarrow \{0,1\}$ as
$$
\xi_T : y \mapsto
\left\{
\begin{array}{lcl}
1 &\qquad & \sum_{R \subseteq T} \sum_{S \in \mathcal{S}: S \supseteq T} \int a_S(y_R,x_{S \setminus R}) dx_{S\setminus R} > \epsilon_2,
\\ 0 & & \mbox{otherwise}.
\end{array}
\right.
$$
Then set $\delta_0:=2^{-2k}$ and for every $T \subseteq [n]$ with $|T| \le k$, define $J_T: X^T \rightarrow \{0,1\}$ as
$$
J_T : y \mapsto
\left\{
\begin{array}{lcl}
1 &\qquad & \max_{R \subseteq T}   \delta_0^{-2|T \setminus R|} \int \xi_T(y_R,x_{T \setminus R}) dx_{T \setminus R}\ge 1,
\\ 0 & & \mbox{otherwise}.
\end{array}
\right.
$$
For $T \subseteq [n]$ with $|T| > k$, define $J_T \equiv 0$. In the sequel it will be useful to bear in mind that for $S \in \mathcal{S}$,
\begin{equation}
\label{eq:basicAboutXi}
1_{[|F_S| > \epsilon_1]} \le \xi_S \le J_S,
\end{equation}
and in general for every $T \subseteq [n]$ with $|T| \le k$, it holds that
\begin{equation}
\label{eq:basicAboutXiII}
\xi_T \le J_T.
\end{equation}

In order to be able to bound $\left\|g - \Ex[g | \mathcal{F}_{\mathcal{J}}] \right\|_2$ and $\int |J_\mathcal{J}|$ we need to prove a few lemmas.
\begin{lemma}
\label{lem:dichot}
Consider a set $T \subseteq [n]$ with $|T| \le k$. For every $R \subseteq T$ and every $y \in X^{R}$, either $\int J_T(y,x_{T \setminus R}) dx_{T \setminus R} \le  \delta_0$, or $J_T(z) = 1$ for every $z \in X^T$ with $z_R=y$.
\end{lemma}
\begin{proof}
Suppose that $J_T(z) = 0$, for at least one $z \in X^T$ with $z_R = y$.  Then by the definition of $J_T$, for every $A \subseteq R$, we have
\begin{equation}
\label{eq:inLemDichot}
\int \xi_T(y_A,x_{T \setminus A}) dx_{T \setminus A} <  \delta_0^{2|T \setminus A|}.
\end{equation}
Hence for every $B \subseteq T$ with $ B \not\subseteq R$, setting $A:=B \cap R$, we have
\begin{eqnarray*}
\int  1_{\left[\int \xi_T(y_A,x_{T \setminus A}) dx_{T \setminus B}\ge \delta_0^{2|T \setminus B|} \right]} dx_{T \setminus R}  &=&
\int  \left( \int 1_{\left[\int \xi_T(y_A,x_{T \setminus A})  dx_{T \setminus B}\ge \delta_0^{2|T \setminus B|} \right]} dx_{T \setminus B} \right) dx_{T \setminus R}
\\ & \le & \int  \left(  \int \xi_T(y_A,x_{T \setminus A})  \delta_0^{-2|T \setminus B|} dx_{T \setminus B} \right) dx_{T \setminus R}
\\ &=& \delta_0^{-2|T \setminus B|} \int  \xi_T(y_A,x_{T \setminus A})   dx_{T \setminus A}  \\
& \le & \delta_0^{-2|T \setminus B|}  \delta_0^{2|T \setminus A|} \le \delta_0^2.
\end{eqnarray*}
Then using (\ref{eq:inLemDichot}) we conclude that
\begin{eqnarray*}
\int J_T(y_R,x_{T \setminus R}) dx_{T \setminus R}  &\le& \int \sum_{A \subseteq R} 1_{\left[\int \xi_T dx_{T \setminus A}\ge \delta_0^{2|T \setminus A|} \right]} dx_{T \setminus R}+ \int \sum_{\substack{B:B \subseteq T,\\ B \not\subseteq R}} 1_{\left[\int \xi_T dx_{T \setminus B}\ge \delta_0^{2|T \setminus B|} \right]} dx_{T \setminus R} \\
&=& \int \sum_{B:B \subseteq T, B \not\subseteq R} 1_{\left[\int \xi_T dx_{T \setminus B}\ge \delta_0^{2|T \setminus B|} \right]} dx_{T \setminus R} = \sum_{B:B \subseteq T, B \not\subseteq R}\delta_0^2   \le   2^k \delta_0^2 \le \delta_0.
\end{eqnarray*}
\end{proof}

The following lemma shows that for a random $x \in X^T$, the expected value of $\sum_{T: |T| \in k} J_T(x)$ is bounded from above by the constant $\epsilon_2^{-2}$.
\begin{lemma}
\label{lem:boundExp}
We have
$$\sum_{T: |T| \le k} \int J_T  \le \epsilon_2^{-2}.$$
\end{lemma}
\begin{proof}
Note that
\begin{eqnarray}
\nonumber
\sum_{T: |T| \le k} \int J_T  dx_T &\le& \sum_{T:|T| \le k} \int  \delta_0^{-2k} \sum_{R \subseteq T} \int \xi_T dx_{T \setminus R}  dx_T
\le \delta_0^{-2k} \sum_{T:|T| \le k} \sum_{R \subseteq T} \int \xi_T dx_T
\\ &\le& 2^k \delta_0^{-2k}  \sum_{T: |T| \le k} \int \xi_T dx_T.
\label{eq:tildaRemove}
\end{eqnarray}
Furthermore by the definition of $\xi_T$ and Part (d) of Lemma~\ref{lem:Properties},
\begin{eqnarray}
\nonumber
\sum_{T: |T| \le k} \int  \xi_T dx_{[n]} &\le& \epsilon_2^{-1}   \sum_{T: |T| \le k} \int \left( \sum_{R \subseteq T}  \sum_{S \in \mathcal{S}, S \supseteq T} \int a_S dx_{[n]\setminus R}\right)  dx_{[n]} \\ \nonumber &\le& 2^k \epsilon_2^{-1} \sum_{T: |T| \le k} \sum_{S \in \mathcal{S}, S \supseteq T}
\int a_S dx_{[n]}  \le  2^{2k}\epsilon_2^{-1} \int \sum_{S \in \mathcal{S}} a_S dx_{[n]}
\\&\le & 2^{2k} \epsilon_2^{-1} \delta^{-3k}.  \label{eq:boundTilde}
\end{eqnarray}
From (\ref{eq:tildaRemove}), (\ref{eq:boundTilde}), and the assumptions  $\epsilon_2 = \delta^{10k}$, we conclude that
$$\int \sum_{T: |T| \le k} J_T dx_T \le 2^{3k} \delta_0^{-2k} \delta^{-3k} \epsilon_2^{-1} \le \epsilon_2^{-2}.$$
\end{proof}

It follows immediately from Lemma~\ref{lem:boundExp} that
\begin{equation}
\label{eq:boundMathcalJ}
\int |J_\mathcal{J}| \le \int \sum_{T \subseteq [n], |T| \le k} |T| J_T \le k \sum_{T: |T| \le k} \int J_T  \le k\epsilon_2^{-2}.
\end{equation}

\subsection{Step V: Bounding the error}
In this section we show that $\left\|g - \Ex[g | \mathcal{F}_{\mathcal{J}}] \right\|_2$ is small.
By linearity of expectation $\Ex[g | \mathcal{F}_{\mathcal{J}}] = \sum_{S \in \mathcal{S}} \Ex[G_S | \mathcal{F}_{\mathcal{J}}]$. However since $\mathcal{F}_{\mathcal{J}}$ depends on all coordinates, it is difficult to analyze $\Ex[G_S | \mathcal{F}_{\mathcal{J}}]$ directly. To remedy this we define some auxiliary $\sigma$-algebras. For every $S \subseteq [n]$, define the collection $\mathcal{J}_S := \{\tilde{J}_T\}_{T \subseteq [n]}$ in the following way. For every $T \subseteq [n]$, we set $\tilde{J}_T := J_T$ if  $T \subseteq S$, and  $\tilde{J}_T \equiv 0$ otherwise. Note that $\mathcal{J} = \mathcal{J}_{[n]}$. Define $\tilde{g} := \sum_{S \in \mathcal{S}} \tilde{G}_S$, where
$\tilde{G}_S:=\Ex[G_S  | \mathcal{F}_{\mathcal{J}_S}]$.  For every $S \subseteq [n]$,
the function $\tilde{G}_S$ depends only on the coordinates in $S$ and furthermore
the $\sigma$-algebra $\mathcal{F}_{\mathcal{J}_S}$ is coarser than the $\sigma$-algebra
$\mathcal{F}_{\mathcal{J}}$. It follows from the latter that $\tilde{g}$ is measurable with respect to $\mathcal{F}_{\mathcal{J}}$ and hence
\begin{equation}
\label{eq:BoundConditional}
\left\|g - \Ex[g | \mathcal{F}_{\mathcal{J}}] \right\|_2 \le \left\|g - \tilde{g}\right\|_2.
\end{equation}
So in order to bound $\left\|g - \Ex[g | \mathcal{F}_{\mathcal{J}}] \right\|_2$, it suffices to bound the right-hand side of (\ref{eq:BoundConditional}). Trivially for every $S  \in \mathcal {S}$,
$$\int G_S \tilde{G}_S =\int \tilde{G}_S^2.$$
For $S_1, S_2 \in \mathcal{S}$, since $\tilde{G}_{S_1}$ and $\tilde{G}_{S_2}$  depend respectively only on the coordinates in $S_1$ and $S_2$, if
$S_1 \not\subseteq S_2$ and $S_2 \not\subseteq S_1$, then
$$\int G_{S_1}G_{S_2} = \int G_{S_1} \tilde{G}_{S_2} =\int  \tilde{G}_{S_1} G_{S_2} = 0.$$
Similarly for $S_1, S_2 \in \mathcal{S}$, if $S_1 \cap S_2 = \emptyset$, then
$$
\int \tilde{G}_{S_1} \tilde{G}_{S_2} = \left(\int \tilde{G}_{S_1} \right)\left(\int \tilde{G}_{S_2}\right) = \left(\int G_{S_1}\right) \left(\int G_{S_2}\right)=0.$$
Thus
\begin{eqnarray}
\nonumber
\|g - \tilde{g}\|_2^2 &=& \int \left(\sum_{S \in \mathcal{S}} G_S - \tilde{G}_S\right)^2 \le (k+1) \sum_{r=0}^k \int \left(\sum_{S \in \mathcal{S}, |S|=r} G_S - \tilde{G}_S\right)^2 \\ \nonumber
&=& (k+1)
\int \sum_{S \in \mathcal{S}} G_S^2 - \tilde{G}_S^2
+ \sum_{\substack{S_1,S_2 \in \mathcal{S},  S_1 \cap S_2 \neq \emptyset \\ S_1 \neq S_2, |S_1|=|S_2|}}\tilde{G}_{S_1} \tilde{G}_{S_2}
\\
&\le & 2k \int \sum_{S \in \mathcal{S}} |G_S^2 - \tilde{G}_S^2| \label{eq:same}
\\
&& + 2k  \sum_{\substack{S_1,S_2 \in \mathcal{S} \\ S_1 \cap S_2 \neq \emptyset, S_1 \neq S_2}}\left|\int \tilde{G}_{S_1} \tilde{G}_{S_2} \right|
\label{eq:distinct}
\end{eqnarray}

\noindent \emph{Bounding (\ref{eq:same}):} Consider an $S \in \mathcal{S}$ and an $x \in X^S$. By (\ref{eq:basicAboutXi}) if $|F_S(x)| > \epsilon_1$, then $J_S(x)=1$ which implies that $J_{\mathcal{J}_S}(x)=S$, and consequently $G_S(x)=\tilde{G}_S(x)$. Hence it follows from Part (a) of Lemma~\ref{lem:Properties} and Parseval's identity  that
\begin{eqnarray}
\nonumber
(\ref{eq:same})&=&2k \int \sum_{S \in \mathcal{S}} \left|G_S^2 - \tilde{G}_S^2\right| = 2k \int \sum_{S \in \mathcal{S}} |G_S -  \tilde{G}_S|^2
=2k \int \sum_{S \in \mathcal{S}} |G_S -  \tilde{G}_S|^2 1_{[|F_S| \le \epsilon_1]}
\\ &\le&  2k \int \sum_{S \in \mathcal{S}}G_S^2 1_{[|F_S| \le \epsilon_1]} \le 4 \epsilon_0 \sum_{S \in \mathcal{S}} \|F_S\|_2^2 \le 4 \epsilon_0.
\label{eq:boundSame}
\end{eqnarray}

\noindent \emph{Bounding (\ref{eq:distinct}):} We start with a lemma about the size of the atoms of the $\sigma$-algebra $\mathcal{F}_{\mathcal{J}_S}$.
\begin{lemma}
\label{lem:fat}
Consider $S \in \mathcal{S}$,  a subset $A \subseteq S$, and an element $y \in X^S$ with $J_{\mathcal{J}_S}(y)=A$. Then
$\int  1_{[J_{\mathcal{J}_S}(y_A,x_{S \setminus A}) = A]}  dx_{S \setminus A} \ge \frac{1}{2}$.
\end{lemma}
\begin{proof}
Since $J_{\mathcal{J}_S}(y)=A$, for every $R \subseteq S$ with $R \not\subseteq A$, we have $J_R(y)=0$. Then it follows from Lemma~\ref{lem:dichot} that for every such $R$, we have
$$\int  J_R(y_{A\cap R},x_{R \setminus A})  dx_{R \setminus A} \le \delta_0.$$
Hence
\begin{eqnarray*}
\int 1_{[J_{\mathcal{J}_S}(y_A,x_{S \setminus A}) = A]}  dx_{S \setminus A} &=& \int \prod_{R \subseteq S, R \not\subseteq A} 1_{[J_R(y_{A\cap R},x_{R \setminus A}) = 0]}    dx_{S \setminus A}\\ &\ge& 1 - \sum_{R \subseteq S, R \not\subseteq A} \int  J_R(y_{A \cap R},x_{R \setminus A})  dx_{S \setminus A} \ge 1 -2^k \delta_0 \ge \frac{1}{2}.
\end{eqnarray*}
\end{proof}

\begin{lemma}
\label{lem:GSdT}
Consider $S \in \mathcal{S}$, and a subset $T \subseteq S$. For every $y \in  X^T$, we have
$$\left|\int \tilde{G}_S(y,x_{S \setminus T})  dx_{S \setminus T} \right| \le 2^{k+1}  \sum_{R \subseteq T }\int a_S(y_R,x_{S \setminus R})   dx_{S \setminus R}. $$
Furthermore if $T \subsetneq S$ and $J_T(y)=1$, then $\int \tilde{G}_S(y,x_{S \setminus T})  dx_{S \setminus T}=0$.
\end{lemma}
\begin{proof}
Consider $A \subseteq S$, and $z \in X^S$ with $J_{\mathcal{J}_S}(z)=A$. By Lemma~\ref{lem:fat} and Part (c) of Lemma~\ref{lem:Properties}, we have
\begin{eqnarray*}
|\tilde{G}_S(z)| &=&  \frac{\left|\int G_S(z_A,x_{S \setminus A}) 1_{[J_{\mathcal{J}_S}(z_A,x_{S \setminus A})=A]} dx_{S \setminus A}\right|}{\int 1_{[J_{\mathcal{J}_S(z_A,x_{S \setminus A})}=A]} dx_{S \setminus A}} \\ &\le&  2 \left|\int G_S(z_A,x_{S \setminus A}) 1_{[J_{\mathcal{J}_S}(z_A,x_{S \setminus A})=A]} dx_{S \setminus A}\right| \\ &\le& 2 \int a_S(z_A,x_{S \setminus A}) 1_{[J_{\mathcal{J}_S}(z_A,x_{S \setminus A})=A]} dx_{S \setminus A} \le  2 \int a_S(z_A,x_{S \setminus A})  dx_{S \setminus A}.
\end{eqnarray*}
Hence for every $y \in X^T$,
\begin{eqnarray*}
\left|\int \tilde{G}_S(y,x_{S \setminus T})  dx_{S \setminus T} \right| &\le& \sum_{A \subseteq S} \left|\int \tilde{G}_S(y,x_{S \setminus T}) 1_{[J_{\mathcal{J}_S}(y,x_{S \setminus T})=A]}  dx_{S \setminus T} \right|
\\ &\le& 2 \sum_{A \subseteq S} \int   a_S(y_{A \cap T},x_{S \setminus (A \cap T)})  dx_{S \setminus (A \cap T)}  \le 2^{k+1}  \sum_{R \subseteq T }\int a_S(y_R,x_{S \setminus R})   dx_{S \setminus R}.
\end{eqnarray*}
This verifies the first assertion of the lemma. Now suppose that $T \subsetneq S$ and $J_T(y)=1$. Then for every $z \in X^S$ with $z_T=y$, we have $J_{\mathcal{J}_S}(z) \supseteq T$. Hence
\begin{eqnarray*}
\int \tilde{G}_S(y,x_{S \setminus T})  dx_{S \setminus T} &=& \sum_{A: T \subseteq A \subseteq S} \int \tilde{G}_S 1_{[J_{\mathcal{J}_S}=A]}  dx_{S \setminus T} =
\sum_{A: T \subseteq A \subseteq S} \int \int \tilde{G}_S 1_{[J_{\mathcal{J}_S}=A]} dx_{S \setminus A} dx_{S \setminus T}
\\&=& \sum_{A: T \subseteq A \subseteq S} \int \int G_S 1_{[J_{\mathcal{J}_S}=A]} dx_{S \setminus A} dx_{S \setminus T} =  \int G_S(y,x_{S \setminus T})  dx_{S \setminus T} =0.
\end{eqnarray*}
\end{proof}

Now consider two distinct $S_1,S_2 \in \mathcal {S}$ with $T:=S_1 \cap S_2 \neq \emptyset$. By the second assertion of Lemma~\ref{lem:GSdT}, for $y \in X^T$, if $J_T(y) = 1$, then
$$\int \tilde{G}_{S_1}(y,x_{[n] \setminus T}) \tilde{G}_{S_2}(y,x_{[n] \setminus T}) dx_{[n] \setminus T}= 0.$$
It follows that
\begin{eqnarray*}
(\ref{eq:distinct}) &=& 2k\sum_{\substack{S_1,S_2 \in \mathcal {S} \\S_1 \neq S_2, S_1 \cap S_2 \neq \emptyset}} \left| \int \tilde{G}_{S_1} \tilde{G}_{S_2}\right|
= 2k \sum_{T \subseteq [n]} \sum_{\substack{S_1,S_2 \in \mathcal {S} \\S_1 \neq S_2, S_1 \cap S_2=T}} \left|\int \tilde{G}_{S_1} \tilde{G}_{S_2} 1_{[J_{T}=0]} \right| \\
&=& 2k\sum_{T \subseteq [n]} \sum_{\substack{S_1,S_2 \in \mathcal {S} \\S_1 \neq S_2, S_1 \cap S_2=T}} \left|\int \left(\int \tilde{G}_{S_1} dx_{[n] \setminus T} \right)\left(\int \tilde{G}_{S_2} dx_{[n] \setminus T}\right)  1_{[J_{T}=0]} dx_{T} \right|
\\
&\le & 2k\sum_{T \subseteq [n]} \int   \left( \sum_{S \in \mathcal{S}, S \supseteq T} \left| \int \tilde{G}_{S} dx_{[n] \setminus T} \right|\right)^2 1_{[J_{T}=0]} dx_T
\end{eqnarray*}
Now applying the first assertion of Lemma~\ref{lem:GSdT}, we have
\begin{equation}
\label{eq:distinctProcessed}
(\ref{eq:distinct}) \le 2k\sum_{T \subseteq [n]} \int   \left(2^{k+1}  \sum_{S \in \mathcal{S}, S \supseteq T}\sum_{R \subseteq T }\int a_S   dx_{[n] \setminus R}\right)^2 1_{[J_{T}=0]} dx_T.
\end{equation}
Recall from (\ref{eq:basicAboutXiII}) that $\xi_T \le J_T$. Hence if $J_{T}(y)=0$, for $y \in X^T$, then
$$ \sum_{R \subseteq T} \sum_{S \in \mathcal{S}: S \supseteq T} \int a_S(y_R,x_{[n]\setminus R}) dx_{[n]\setminus R} \le \epsilon_2.$$
So (\ref{eq:distinctProcessed}) together with Part (d) of Lemma~\ref{lem:Properties} implies that
\begin{eqnarray}
\nonumber
(\ref{eq:distinct}) &\le&  2k\sum_{T \subseteq [n]} 2^{3k} \epsilon_2 \int \left( \sum_{S \in \mathcal{S}, S \supseteq T} \sum_{R \subseteq T }\int a_S   dx_{[n] \setminus R} \right)dx_{T} \\ &\le&  2k 2^{3k} \epsilon_2 \sum_{T \subseteq [n]} \sum_{R \subseteq T } \sum_{\substack{S \in \mathcal{S}\\ S \supseteq T}} \int  a_S
\le 2k 2^{5k} \epsilon_2  \sum_{S \in \mathcal{S}}  \int  a_S   dx_{[n]} \le 2k 2^{5k} \epsilon_2 \delta^{-3k}
\le \epsilon_0.
\label{eq:boundDistinct}
\end{eqnarray}
Combining (\ref{eq:BoundConditional}), (\ref{eq:boundSame})  and (\ref{eq:boundDistinct}), we conclude that
\begin{equation}
\label{eq:boundProjectg}
\left\|g - \Ex[g | \mathcal{F}_{\mathcal{J}}] \right\|^2_2 \le \|g - \tilde{g}\|_2^2 \le 5 \epsilon_0.
\end{equation}

\subsection{Step VI: Finishing the proof\label{sec:Step6}}

By (\ref{eq:cutFourier}) and (\ref{eq:proningSmall}), we have
$$\int \left|f - \sum_{S \in \mathcal{S}} F_S \right|^2 \le 2 \epsilon_0.$$
Recall that $g=\sum_{S \in \mathcal{S}} G_S$. By Part (b) of Lemma~\ref{lem:Properties},
$$\int \left|\sum_{S \in \mathcal{S}} F_S - \sum_{S \in \mathcal{S}} G_S \right|^2 =  \sum_{S \in \mathcal{S}} \int \left| F_S - G_S\right|^2
\le \sum_{S \in \mathcal{S}} 10 \epsilon_0 k^{-1} \|F_S\|_2^2 \le 10 k^{-1}\epsilon_0 \le \epsilon_0.$$
Hence
$$\int |f - g|^2 \le 2 \int \left|f - \sum_{S \in \mathcal{S}} F_S\right|^2 +  \left|\sum_{S \in \mathcal{S}} F_S - \sum_{S \in \mathcal{S}} G_S\right|^2 \le  4 \epsilon_0 + 2 \epsilon_0 \le 6 \epsilon_0.$$
Then by (\ref{eq:boundProjectg}),
$$\int |f-\Ex[f|\mathcal{F}_\mathcal{J}|^2 \le \int |f-\Ex[g | \mathcal{F}_{\mathcal{J}}]|^2 \le 2 \int |f-g|^2 + |g-\Ex[g | \mathcal{F}_{\mathcal{J}}]|^2 \le 12 \epsilon_0 + 10\epsilon_0 \le 22 \epsilon_0.$$
Define $h:X^n \rightarrow \{0,1\}$ as
$$h(x) :=
\left\{
\begin{array}{lcl}
0 & \qquad & \Ex[f|\mathcal{F}_\mathcal{J}]  \le \frac{1}{2}\\
1 & & \Ex[f|\mathcal{F}_\mathcal{J}]  > \frac{1}{2}.
\end{array}
\right.
$$
Note that $h$ is a Boolean function and it is measurable with respect to $\mathcal{F}_\mathcal{J}$. Since $f$ is a Boolean function, we have
\begin{eqnarray*}
\|f-h\|_1  &=& \int |f-h|^2\le \int 4|f-\Ex[f|\mathcal{F}_\mathcal{J}]|^2 \le 90 \epsilon_0  \le \epsilon.
\end{eqnarray*}
We also verified in (\ref{eq:boundMathcalJ}) that
$\int |J_\mathcal{J}(x)| dx\le  k\epsilon_2^{-2} \le e^{10^{15}  \epsilon^{-3} \lceil I_f \rceil ^3}.$

\section{Concluding remarks\label{sec:concl}}
\begin{itemize}

\item For graph properties which are also  \emph{increasing} (as it is the case when one studies the threshold phenomenon), what is shown by Friedgut~\cite{MR1678031}  is slightly stronger than the assertion of Theorem~\ref{thm:mainpbiased}. He showed that in this case, one can deduce from the structure provided in Theorem~\ref{thm:mainpbiased} that the property is essentially generated by small minimal elements. He then conjectured this statement (Conjecture 1.5 in~\cite{MR1678031}) for general increasing set properties. Friedgut's conjecture remains unsolved as we do not know how to refine the structure provided in Theorem~\ref{thm:mainpbiased} for increasing functions without appealing to any symmetry assumptions.

\item In~\cite{BKKKL} a different notion of the influence is defined, and in~\cite{MR2501432} the structure of Boolean functions that have small total influences with the definition of~\cite{BKKKL} is determined.
\end{itemize}

\section*{Acknowledgements}
The author is grateful to Serguei Norine for pointing the connection to the FKG inequality  which led to the proof of Corollary~\ref{cor:monotone}. He also wishes to thank him and Pooya Hatami for many valuable discussions, and Elchanan Mossel and Mokshay Madiman for pointing out the references~\cite{MR615434,MR0026294}.

\bibliographystyle{alpha}
\bibliography{pseudojunta}

\end{document}